\documentclass[11pt]{amsart}

\usepackage[T1]{fontenc}
\usepackage{lmodern}
\usepackage{amsmath,amssymb,amsthm,mathtools,amscd}
\usepackage[hidelinks]{hyperref}
\usepackage{microtype}
\numberwithin{equation}{section}

\newtheorem{theorem}{Theorem}[section]
\newtheorem{proposition}[theorem]{Proposition}
\newtheorem{lemma}[theorem]{Lemma}
\newtheorem{corollary}[theorem]{Corollary}
\theoremstyle{definition}
\newtheorem{definition}[theorem]{Definition}
\theoremstyle{remark}
\newtheorem{remark}[theorem]{Remark}

\newcommand{\T}{\mathbb T}
\newcommand{\C}{\mathbb C}
\newcommand{\Z}{\mathbb Z}

\newcommand{\rank}{\operatorname{rank}}
\newcommand{\diag}{\operatorname{diag}}

\newcommand{\SP}{\operatorname{SP}}
\newcommand{\SPV}{\operatorname{SPV}}

\title{Complexity Rank of AT Algebras of Real Rank Zero}

\author{Qingnan An}
\address{School of Mathematics and Statistics, Northeast Normal University,\newline Changchun, {\rm130024}, China}
\email{anqn645@nenu.edu.cn}

\author{Kang Li}
\address{Centre for Mathematical Sciences, Lund University, Box 118, SE-221 00 Lund, Sweden}
\email{kang.li@math.lth.se}

\author{Zhichao Liu}
\address{School of Mathematical Sciences,
Dalian University of Technology,
Dalian, {\rm116024}, China}
\email{zhichaoliu@dlut.edu.cn}

\date{}
\hypersetup{
  pdftitle={Complexity Rank of AT Algebras of Real Rank Zero},
  pdfsubject={A proof using cyclic models, spectral variation, and KK-shape equivalence},
  pdfkeywords={AT algebra, complexity rank, real rank zero, spectral variation}
}

\begin{document}

\begin{abstract}
In this paper, we prove that every separable unital AT algebra of real rank zero has complexity rank at most one.
In particular, every irrational rotation algebra has complexity rank one.  This answers a question of Jaime and Willett.
\end{abstract}

\maketitle

\section{Introduction}
The notion of complexity rank for $C^*$-algebras was introduced by Willett and Yu in \cite{WY} as part of their study of the UCT problem. This notion was motivated by finite decomposition complexity in coarse geometry \cite{GTY}, dynamical
complexity \cite{GWY}, and nuclear dimension \cite{WZ}. In the same paper \cite[Theorem~1.7]{WY}, Willett and Yu proved that the UCT problem is equivalent to the statement that every Kirchberg algebra with zero $K$-theory has complexity rank one.

The first nontrivial computation of complexity rank was carried out by Willett and Yu, who showed that the Cuntz algebras $\mathcal O_n$, for
$2\leq n<\infty$, have complexity rank one
\cite[Proposition~A.1]{WY}. Subsequently, Jaime and Willett proved that every
unital UCT Kirchberg algebra has complexity rank either one or two, and that it has complexity rank one precisely when its $K_1$-group is torsion-free \cite[Theorem~1.6]{JW}. These results provide a fairly broad picture in the purely infinite setting, whereas much less is known about complexity rank in the stably finite setting.

In \cite[Question~6.2]{JW}, Jaime and Willett asked about the complexity rank of irrational rotation algebras and conjectured that every non-AF separable AT algebra of real rank zero has complexity rank one. AT algebras provide a natural testing ground for the study of complexity rank. They are inductive limits of finite direct sums of matrix algebras over \(C(\mathbb T)\), possibly together with finite-dimensional summands, and form
one of the earliest successful classes in the Elliott classification program. Historically, Elliott classified AT algebras of real rank zero by their graded ordered scaled $K$-theory \cite{Ell}.
Bratteli and Elliott subsequently characterized real rank zero for inductive limits of finite direct sums of matrix algebras over finite-dimensional compact metrizable spaces by very small eigenvalue variation \cite{BE}. In the work of Elliott--Gong, Dadarlat, and Dadarlat--Gong on the classification of AH algebras of real rank zero, spectral and weak variation provided quantitative control of connecting maps and enabled approximate decompositions involving homomorphisms with finite-dimensional image \cite{EG,D,DG}.
These ideas provide important structural input for our arguments in this paper.

Recently, An and Liu proved that every unital $C^*$-algebra of complexity rank at most one has real rank zero \cite[Theorem~4.1]{AL}. In this paper, we prove the converse implication for AT algebras, thereby answering
a question of Jaime and Willett \cite[Question~6.2]{JW}:

\begin{theorem}\label{thm:main}
Let $A$ be a separable unital AT algebra of real rank zero.  Then
the complexity rank of $A$ is at most one.
\end{theorem}
By definition, a separable unital $C^*$-algebra has complexity rank zero if and only if it is AF. Thus, every non-AF algebra covered by Theorem~\ref{thm:main} has complexity rank one. The most relevant example is the irrational rotation algebra $A_\theta$ for an \emph{irrational} number $\theta$. More precisely, $A_\theta$ is the universal unital (separable) $C^*$-algebra generated by
unitaries $u$ and $v$ satisfying
$
 vu=e^{2\pi i\theta}uv.
$
In \cite{EE}, Elliott and Evans proved that every irrational rotation
algebra $A_\theta$ is an AT algebra of real rank zero. Moreover, $K_1(A_\theta)\cong\mathbb Z^2$ by \cite{PV}. Therefore, it cannot be AF. As a consequence of Theorem~\ref{thm:main}, we obtain the following corollary, which answers \cite[Question 6.2]{JW}:
\begin{corollary}
Every irrational rotation algebra $A_\theta$ has complexity rank one.
\end{corollary}

The proof of Theorem~\ref{thm:main} combines cyclic model maps with a two-cut
construction. Small spectral variation forces large multiplicity for
circle-to-circle homomorphisms of nonzero $K_1$-degree
(Lemma~\ref{lem:spectral-cover}). We replace the circle-to-circle partial maps
by $KK$-equivalent cyclic models whose spectral variation is controlled by
multiplicity, independently of degree (Lemma~\ref{lem:full-model}). For these
models, two separated cuts in the cyclic shift yield an approximately central
positive contraction and three finite-dimensional subalgebras with an exact
common overlap; the estimates are uniform in degree and remain valid after
composition with outer homomorphisms (Lemma~\ref{lem:two-cut}).
Proposition~\ref{prop:local-model} assembles these decompositions. After passing
to a subsequence and replacing the connecting maps, we arrange that composites
from each fixed stage have spectral variation tending to zero, ensuring real
rank zero for the new limit (Theorem~\ref{thm:small-variation}). The resulting
system is $KK$-shape equivalent to the original one, so
Proposition~\ref{prop:KK-shape} identifies their limits. The finite-dimensional
criterion of Proposition~\ref{prop:criterion} then gives complexity rank at
most one.

\section{Preliminaries}

We fix notation and recall the finite-dimensional criterion for complexity
rank at most one. We then review spectral variation and the classification
of AT algebras of real rank zero.

All homomorphisms in this paper are $*$-homomorphisms, and all tensor
products are spatial. We write $M_s=M_s(\C)$ and use $e_{ab}$ for the
standard matrix units. The unit of a corner $pBp$ is $p$: thus a
homomorphism into this corner is unital when it sends the source unit
to $p$. Whenever we work in a corner, expressions such as $1-h$ are
understood as $p-h$. For an inductive system, $\phi_{n,m}$ denotes the
map from stage $n$ to stage $m$, with
$\phi_{n,r}=\phi_{m,r}\phi_{n,m}$ for $n<m<r$.

\begin{definition}[\cite{BP}]
A unital $C^*$-algebra has \emph{real rank zero} if its
self-adjoint elements with finite spectrum are norm dense in its self-adjoint part.
\end{definition}

\begin{definition}
Let $A$ be a $C^*$-algebra. Denote by $A_{sa}$ its self-adjoint part.
For $a,b\in A$ and $\varepsilon>0$, write $a\approx_\varepsilon b$
if $\|a-b\|<\varepsilon$. For a subset $S\subseteq A$, write
$a\in_\varepsilon S$ if $a\approx_\varepsilon s$ for some $s\in S$.
We use $[a,b]$ to denote $ab-ba$.
\end{definition}

\begin{definition}
Let $\mathcal{C}$ be a class of $C^*$-algebras. A $C^*$-algebra $A$ is locally in $\mathcal{C}$ if, for any finite subset $X$ of $A$ and any $\varepsilon>0$, there is a $C^*$-subalgebra $C$ of $A$ that belongs to $\mathcal{C}$ and satisfies $x \in_\varepsilon C$ for all $x \in X$.
\end{definition}

\begin{definition}[{\cite[Definition~1.1]{WY}, \cite{JW}}]\label{def:complexity}
Let $A$ be a unital $C^*$-algebra, and let $\mathcal{C}$ be a class of unital $C^*$-algebras. Then $A$ decomposes over $\mathcal{C}$ if, for every finite subset $X$ of the unit ball of $A$ and every $\varepsilon>0$, there exist $C^*$-subalgebras $C, D$, and $E$ of $A$ that are in the class $\mathcal{C}$ and contain $1_A$, and a positive contraction $h \in E$ such that

(i) $\|[h, x]\|<\varepsilon$ for all $x \in X$;

(ii) $h x \in_\varepsilon C,\left(1_A-h\right) x \in_\varepsilon D$, and $h\left(1_A-h\right) x \in_\varepsilon E$ for all $x \in X$;

(iii) $e\in_\varepsilon C$ and $e\in_\varepsilon D$ for every $e$ in the unit ball of $E$.
\end{definition}

\begin{definition}[{\cite[Definition~2.3]{JW}}]
Let $\alpha$ be an ordinal number.

(i) If $\alpha=0$, let $\mathcal{D}_0$ be the class of unital $C^*$-algebras that are locally finite-dimensional.

(ii) If $\alpha>0$, let $\mathcal{D}_\alpha$ be the class of unital $C^*$-algebras that decompose over $C^*$-algebras in $\bigcup_{\beta<\alpha} \mathcal{D}_\beta$.

A unital $C^*$-algebra has \emph{finite complexity} if it is in $\mathcal{D}_\alpha$ for some $\alpha$, in which case its \emph{complexity rank} is the smallest possible $\alpha$. We say a unital $C^*$-algebra has complexity rank zero if and only if it is locally finite-dimensional.

\end{definition}
\begin{remark}
  The class of $C^*$-algebras $\mathcal{D}_\alpha$ is closed under taking finite direct sums, quotients and inductive limits with unital connecting maps. Any unital $C^*$-algebra that is locally in $\mathcal{D}_\alpha$
is in $\mathcal{D}_\alpha$; see  \cite[Lemmas 2.5--2.7]{JW}.
\end{remark}

\begin{proposition}[{\cite[Proposition 2.16]{JW}}]
\label{prop:criterion}
A unital $C^*$-algebra $A$ has complexity rank at most one if and only if the
following condition holds.  For every finite subset $X$ of the unit ball of
$A$ and every $\varepsilon>0$, there are finite-dimensional unital
$C^*$-subalgebras $C,D,E\subseteq A$ and a positive contraction $h\in E$
such that $E\subseteq C\cap D$ and, for every $x\in X$,
$$
 \|[h,x]\|<\varepsilon,
 \qquad
 hx\in_\varepsilon C,
 \qquad
 (1-h)x\in_\varepsilon D,
 \qquad
 h(1-h)x\in_\varepsilon E.
$$
\end{proposition}

In this criterion, the three algebras have the same unit as $A$.
The inclusion $E\subseteq C\cap D$ makes condition~(iii) of
Definition~\ref{def:complexity} exact. The term $h(1-h)x$ records the
part where both pieces are present, and must be approximated inside
this common algebra. This is why the construction below keeps track
of $E$, in addition to approximating $hx$ and $(1-h)x$ separately.

\begin{definition}\label{def:AT-RR0}
A \emph{circle block} is an algebra $M_s(C(\T))$.  A \emph{circle algebra}
is a finite direct sum of circle blocks and finite-dimensional full matrix
algebras.  A separable unital $C^*$-algebra $A$ is an \emph{AT algebra} if $A$ can be written as an inductive limit of circle algebras
$$
 A=\varinjlim(A_n,\phi_{n,m}),
$$
where every connecting homomorphism $\phi_{n,m}$ is
unital.

\end{definition}

\begin{definition}\label{def:partial-map}
Let
$$
 B=\bigoplus_{i=1}^{l}B_i,
 \qquad
 C=\bigoplus_{j=1}^{q}C_j,
$$
where $B_i$ and $C_j$ are unital $C^*$-algebras, and let $\phi:B\to C$ be a homomorphism.  If $\iota_i:B_i\to B$ is the
canonical inclusion and $\pi_j:C\to C_j$ is the canonical projection, then
the $(i,j)$-\emph{partial homomorphism} of $\phi$ is
$$
 \phi^{i,j}=\pi_j\circ\phi\circ\iota_i:B_i\longrightarrow C_j.
$$
Set $p_{ij}=\phi^{i,j}(1_{B_i})$.  Thus
$\phi^{i,j}$ is unital as a map into $p_{ij}C_jp_{ij}$.  For each fixed
$j$, the projections $p_{1j},\ldots,p_{lj}$ are mutually orthogonal; if
$\phi$ is unital, then their sum is $1_{C_j}$.
\end{definition}

Spectral variation was introduced for homogeneous building blocks by Elliott and Gong \cite[Section~1.4]{EG}; the formulation below follows their
matrix-corner convention, as recorded explicitly in
\cite[Sections~1.2.2--1.2.7 and 1.2.16]{Gon}.

\begin{definition}\label{def:spectrum-spv}
Equip $\T$ with the chordal metric $d_{\T}(z,w)=|z-w|$. Let $P\in M_{k_1}(C(\T))$ be a projection of
constant rank $k\geq1$, and let
$$
 \psi:C(\T)\longrightarrow PM_{k_1}(C(\T))P
$$
be a unital homomorphism. Put $P_z=P(z)$. For each $z\in\T$,
there are points $x_1(z),\ldots,x_k(z)\in\T$ and a unitary
$U_z\in M_{k_1}(\C)$ such that
$$
\psi(f)(z)=P_zU_z
\left(\begin{array}{cccc}
                   f(x_1(z)) &  &  &   \\
                    & \ddots &  &    \\
                    &  &  f(x_k(z))&  \\
                    &  &  & \begin{array}{ccc}
                              0 &  &  \\
                               & \ddots &  \\
                               &  & 0
                            \end{array}
                    \\
                 \end{array}
                 \right)
U_z^*P_z.
$$

Write
$$
\operatorname{SP}\psi_z
 =\{x_1(z),\ldots,x_k(z)\}\,\,\text{(counting multiplicity)}.$$
  We call $\operatorname{SP}\psi_z$ the \emph{spectrum} of $\psi$
at the point $z$.
The \emph{spectrum of the homomorphism} is
$$
 \operatorname{SP}(\psi)
 =\bigcup_{z\in\T}\operatorname{SP}\psi_z\subseteq\T,
$$
where multiplicities are omitted in the union.

Following \cite[Definition~1.4.2]{EG}, the
\emph{spectral variation} of $\psi$ is
$$
 \operatorname{SPV}(\psi)
 =
 \sup_{z,z'\in \T}
 \min_{\sigma\in S_k}
 \max_{1\leq i\leq k}
 d_{\T}\bigl(x_i(z),x_{\sigma(i)}(z')\bigr),
$$
where $S_k$ denotes the permutation group on $\{1,\ldots,k\}$, and
$\{x_i(z)\}_{i=1}^k$ and $\{x_i(z')\}_{i=1}^k$ are the spectra of
$\psi$ at $z$ and $z'$, respectively.

We write $d_k$ for the matching metric on unordered multisets
of $k$ points in $\T$, counted with multiplicity:
$$
 d_k\bigl(\{x_1,\ldots,x_k\},\{y_1,\ldots,y_k\}\bigr)
 :=
 \min_{\sigma\in S_k}\max_{1\leq i\leq k}
 |x_i-y_{\sigma(i)}|.
$$
This distance is independent of the enumerations of the multisets. The map
$z\mapsto\SP\psi_z$ is continuous for this metric.
The points in a fiber are unordered, and the minimizing permutation
may depend on the pair of fibers. In particular, the definition does
not require a continuous labeling of the individual eigenvalues
around the circle. This distinction is useful for the cyclic models
in Section~3, whose eigenvalues can be permuted after one circuit.

For a unital homomorphism
$$
 \psi:M_l(C(\T))\longrightarrow PM_{k_1}(C(\T))P,
$$
let
$$
 \psi_1:C(\T)\longrightarrow
 \psi(e_{11}\otimes1_{\T})M_{k_1}(C(\T))\psi(e_{11}\otimes1_{\T}),
 \qquad
 \psi_1(f)=\psi(e_{11}\otimes f).
$$
As in \cite[Section~1.4.3]{EG}, define
$$
 \operatorname{SP}\psi_z
 :=\operatorname{SP}(\psi_1)_z,
 \qquad
 \operatorname{SP}(\psi):=\operatorname{SP}(\psi_1),
 \qquad
 \operatorname{SPV}(\psi)
 :=\operatorname{SPV}(\psi_1).
$$
Thus $\SP\psi_z$ contains $\rank(\psi_1(1))$ points, counted with
multiplicity.

We call $\#(\operatorname{SP}\psi_z)$ the \emph{multiplicity} of $\psi$. Then $$
\#(\operatorname{SP}\psi_z)=
\#(\operatorname{SP}(\psi_1)_z)=\rank(\psi_1(1)).$$

If $B$ and $C$ are circle algebras and $\phi:B\to C$ is a homomorphism,
then $\operatorname{SPV}(\phi)$ is the maximum of
$\operatorname{SPV}(\phi^{i,j})$ over all nonzero circle-to-circle partial
homomorphisms.  A partial homomorphism involving a finite-dimensional block
is assigned spectral variation zero; if there is no nonzero circle-to-circle partial
homomorphism, the maximum is defined to be zero.

\end{definition}

\begin{theorem}[{\cite[Theorem~1.4.5 and Corollary~2.25]{EG}}]
\label{thm:small-variation}
Let $A$ be a unital AT algebra with presentation
$$
 A=\varinjlim(A_n,\phi_{n,m}).
$$
Then $A$ has real rank zero if and only if, for every $n$ and
$\delta>0$, there is $m>n$ such that
$$
 \SPV(\phi_{n,m})<\delta.
$$
In this case, $\SPV(\phi_{n,r})\leq\SPV(\phi_{n,m}) <\delta$ for every $r\geq m$.
\end{theorem}

We also recall the invariant used in the classification theorem.  Its
graded order is needed when the algebra is not simple.

\begin{definition}[{\cite[Section~1.2]{EG}}]\label{def:graded-invariant}
For a unital $C^*$-algebra $B$, put $K_*(B)=K_0(B)\oplus K_1(B)$ and
$$
 K_*(B)^+
 =\left\{\bigl([p],[u+(1_r-p)]\bigr):
 \begin{array}{l}
 p\in M_r(B)\text{ is a projection},\ r\geq1,\\
 u\in U(pM_r(B)p)
 \end{array}\right\}.
$$
The \emph{scaled ordered $K_*$-group} of $B$ is
$$
 \bigl(K_*(B),K_*(B)^+,([1_B],0)\bigr),
$$
with its $\mathbb Z/2$-grading.  For a circle block, in the
standard $K$-theory coordinates,
$$
 K_*(M_s(C(\T)))^+
 =\{(a,b)\in\Z^2:a>0\}\cup\{(0,0)\},
 \qquad ([1],0)=(s,0).
$$
For $M_s$, the invariant is $(\Z,\Z_{\geq0},s)$, concentrated in
degree zero.  Finite direct sums have the coordinatewise graded order.
\end{definition}

\begin{theorem}[Elliott {\cite{Ell}; see also \cite[Theorem~4.6]{EG}}]
\label{thm:AT-classification}
Two separable unital AT algebras of real rank zero are isomorphic if and
only if their scaled ordered $K_*$-groups are isomorphic, with the grading
and order unit preserved.
\end{theorem}

The scaled ordered $K_*$-group is continuous for unital inductive limits.
In particular, two systems with the same building blocks and the same
induced $K_0$ and $K_1$ maps have isomorphic scaled ordered $K_*$-groups
at the limit.

\section{Homomorphisms with prescribed K-theory}

We first relate nonzero degree and small spectral variation to multiplicity.
We then construct cyclic model maps with the same $K$-theory as the given
map on each pair of blocks.

\begin{definition}
Let $p\in M_t(C(\T))$ be a projection of constant rank $sN$, and let
$$
 \psi:M_s(C(\T))\longrightarrow pM_t(C(\T))p
$$
be a unital homomorphism.  Let
$\zeta\in C(\T)$ be the coordinate function with $\zeta(z)=z$.  The degree of
$\psi$ is
$$
 d_{\psi}
 =\operatorname{wind}\det\!\left(
 \psi(\diag(\zeta,1,\ldots,1))+(1_{M_t(C(\T))}-p)
 \right).
$$
Equivalently, the induced $K_1$-map is
$$
K_1(\psi)\,:\,  K_1(M_s(C(\T)))\cong\mathbb{Z}\xrightarrow{\times d_{\psi}}K_1(M_t(C(\T)))\cong\mathbb{Z}.
$$
For the zero homomorphism, we set $N=d_{\psi}=0$.
\end{definition}

We use the standard identifications
$$
 K_0(M_s(C(\T)))\cong\Z,
 \qquad
 K_1(M_s(C(\T)))\cong\Z,
$$
with generators $[e_{11}]$ and $[\diag(\zeta,1,\ldots,1)]$, respectively.
Note that
the matrix-valued Laurent polynomials are dense in $M_s(C(\T))$.

\begin{lemma}\label{lem:spectral-cover}
Let
$$
 \psi:M_s(C(\T))\longrightarrow M_{sN}(C(\T))
$$
be a unital homomorphism with nonzero degree.  Let $0<\delta<1$ and
suppose that $\SPV(\psi)<\delta$.  Then
$$
 N\geq\frac{\pi}{2\arcsin(\delta/2)}.
$$
\end{lemma}

\begin{proof}
For each $z\in\T$, write
$$
 \SP\psi_z=\{x_1(z),\ldots,x_N(z)\}\,\,\text{(counting multiplicity)}.
$$
Let $U\in M_{sN}(C(\T))$ be defined by
$$
 U=\psi\left(\left(
 \begin{array}{ccc}
   \zeta&&\\
   &\ddots&\\
   &&\zeta
 \end{array}
 \right)\right).
$$
Since $\zeta(z)=z$, the eigenvalues of $U(z)$ are the points
$x_1(z),\ldots,x_N(z)$, each repeated $s$ times.

We prove that $\SP(\psi)=\T$.  First,
$$
 \SP(\psi)=\bigcup_{z\in\T}\SP\psi_z
          =\bigcup_{z\in\T}\sigma(U(z))=\sigma(U)\subseteq\T.
$$
Indeed, $U-\lambda1_{sN}$ is invertible if and only if every
$U(z)-\lambda1_{sN}$ is invertible: the pointwise inverses form a
continuous matrix-valued function.  Hence $\SP(\psi)$ is compact.

Suppose that $\lambda\in\T\setminus\sigma(U)$.  A continuous argument
on $\T\setminus\{\lambda\}$ and functional calculus give a self-adjoint
element $h\in M_{sN}(C(\T))$ such that $U=e^{ih}$.  The path
$t\mapsto e^{ith}$ joins $1_{sN}$ to $U$, so
$$
 [U]=0\quad\text{in }K_1(M_{sN}(C(\T)))\cong\Z.
$$
On the other hand,
$$
 [1_s\otimes\zeta]
 =s[\diag(\zeta,1,\ldots,1)]
 \quad\text{in }K_1(M_s(C(\T))),
$$
and therefore
$$
 [U]=K_1(\psi)([1_s\otimes\zeta])=s d_\psi\neq0.
$$
This is a contradiction.

Now fix a base point $z_0\in\T$.  Let $\lambda\in\T$.  There are
$z\in\T$ and $r\in\{1,\ldots,N\}$ such that $x_r(z)=\lambda$.
Since $\SPV(\psi)<\delta$, there is a permutation $\sigma\in S_N$ with
$$
 |x_k(z)-x_{\sigma(k)}(z_0)|<\delta
 \qquad(1\leq k\leq N).
$$
Taking $k=r$ gives
$$
 |\lambda-x_{\sigma(r)}(z_0)|<\delta.
$$
Thus $\{x_1(z_0),\ldots,x_N(z_0)\}$ is a chordal $\delta$-net in $\T$.

If $\mu=e^{i\beta}\in\T$, then
$$
 |e^{i\alpha}-\mu|<\delta
 \quad\Longleftrightarrow\quad
 \left|\sin\frac{\alpha-\beta}{2}\right|<\frac{\delta}{2}.
$$
Each chordal ball of radius $\delta$ intersects $\T$ in an arc of
angular length $4\arcsin(\delta/2)$.  The $N$ arcs cover $\T$, so
$$
 4N\arcsin(\delta/2)\geq2\pi.
$$
\end{proof}

For $N\geq1$ and $d\in\Z$, define a unitary
$$
 S_{N,d}(z)=\sum_{k=1}^{N-1}e_{k,k+1}+z^d e_{N,1}
 =
 \begin{pmatrix}
 0&1&0&\cdots&0\\
 0&0&1&\ddots&\vdots\\
 \vdots&\ddots&\ddots&\ddots&0\\
 0&\cdots&0&0&1\\
 z^d&0&\cdots&0&0
 \end{pmatrix}.
$$
When $N=1$, the sum is empty and $S_{1,d}(z)=z^d$.

Moreover,
$$
 \det S_{N,d}(z)=(-1)^{N-1}z^d.
$$

\begin{definition}\label{def:cyclic-map}
Let $N,s\geq1$ and $d\in\Z$.  Define the \emph{standard cyclic
homomorphism} $\Psi_{N,d}^{(s)}$ by
\begin{align*}
 \Psi_{N,d}^{(s)}:M_s\otimes C(\T)&\longrightarrow M_s\otimes M_N(C(\T)),\\
 a\otimes1_{\T}&\longmapsto a\otimes1_N,\\
 1_s\otimes\zeta&\longmapsto1_s\otimes S_{N,d}.
\end{align*}

 In particular, if
$P=\sum_{|\ell|\leq L}a_\ell\otimes\zeta^\ell$ with each $a_\ell\in M_s$, then
$$
 \Psi_{N,d}^{(s)}(P)
 =\sum_{|\ell|\leq L}a_\ell\otimes S_{N,d}^{\ell}.
$$
\end{definition}
The induced $K$-theory maps are as follows.
\begin{proposition}\label{prop:cyclic-K}
The map $\Psi_{N,d}^{(s)}$ induces multiplication by $N$ on $K_0$ and
multiplication by $d$ on $K_1$.
\end{proposition}

\begin{proof}
The generator $[e_{11}\otimes1]$ of the source $K_0$-group is sent to
$[e_{11}\otimes1_N]$. This projection has rank $N$, so its class is
$N$ times the rank-one generator of the target $K_0$-group.
For $K_1$, apply the map to the standard generator
$u=\diag(\zeta,1,\ldots,1)$. Its image is
$$
 \Psi_{N,d}^{(s)}(u)
 =e_{11}\otimes S_{N,d}+(1_s-e_{11})\otimes1_N.
$$
Under the determinant description of
$K_1(M_{sN}(C(\T)))\cong\Z$, its class is the winding number of
$\det S_{N,d}(z)=(-1)^{N-1}z^d$, which is $d$.

Thus, the two integers have distinct roles: $N$ counts multiplicity,
whereas $d$ measures winding.
\end{proof}

The following lemma fixes the tensor coordinates used in the two-cut
calculation; see \cite[Section~1.3.4]{EG}.

\begin{lemma}\label{lem:straightening}
Let
$$
 \alpha:M_k(C(\T))\longrightarrow pM_t(C(\T))p
$$
be a unital homomorphism, and suppose that $\rank(p)=kR$. Then there is
a $C(\T)$-linear identification
$$
 pM_t(C(\T))p\cong M_k\otimes M_R(C(\T))
$$
under which
$$
 \alpha(a\otimes1_{\T})=a\otimes1_R,
 \qquad
 \alpha(1_k\otimes\zeta)=1_k\otimes V
$$
for every $a\in M_k$ and for some unitary $V\in M_R(C(\T))$.
\end{lemma}

\begin{proof}
Put $f_{ab}=\alpha(e_{ab}\otimes1)$ and $q=f_{11}$.
The projections $f_{aa}$ are mutually equivalent, are orthogonal,
and sum to $p$. Consequently $q$ has constant rank $R$.
Choose a continuous orthonormal frame for its range, written as a
rectangular matrix $v\in M_{t,R}(C(\T))$ with
$v^*v=1_R$ and $vv^*=q$.

For completeness, such a frame can be constructed by pulling $q$
back to the interval $[0,1]$. A continuous frame exists on the
interval. Its endpoint frames differ by a unitary $u\in U(R)$.
Multiplying the frame by a path from $1_R$ to $u^*$ makes the
endpoint frames agree, so the corrected frame descends to $\T$.
Here we use that $U(R)$ is path connected.

For $1\leq a\leq k$, set $v_a=f_{a1}v$ and form the rectangular
matrix $W=(v_1\ \cdots\ v_k)$. The matrix-unit relations give
$$
 v_a^*v_b=\delta_{ab}1_R,\qquad
 W^*W=1_{kR},\qquad WW^*=p.
$$
Hence $b\mapsto W^*bW$ is a $C(\T)$-linear isomorphism from
$pM_t(C(\T))p$ onto $M_k\otimes M_R(C(\T))$, and
$$
 W^*f_{ab}W=e_{ab}\otimes1_R.
$$
The unitary $\alpha(1_k\otimes\zeta)$ commutes with every $f_{ab}$.
After this identification it belongs to the commutant of
$M_k\otimes1_R$, which is $1_k\otimes M_R(C(\T))$.
It therefore has the required form $1_k\otimes V$, with $V$ unitary.
\end{proof}

\begin{lemma}\label{lem:full-model}
Let $\phi:M_s(C(\T))\to M_{sN}(C(\T))$ be a unital homomorphism of
degree $d$. Then
$$
 KK(\phi)=KK(\Psi_{N,d}^{(s)}),
 \qquad
 \SPV(\Psi_{N,d}^{(s)})\leq2\sin\frac{\pi}{2N}.
$$
If $d=0$, the cyclic map has finite-dimensional range and spectral
variation zero.
\end{lemma}

\begin{proof}
Both maps induce multiplication by $N$ on $K_0$ and by $d$ on $K_1$,
by the definition of degree and Proposition~\ref{prop:cyclic-K}.
Since
$$
 K_0(M_s(C(\T)))\cong\Z,
 \qquad K_1(M_s(C(\T)))\cong\Z,
$$
the $\operatorname{Ext}^1$ terms in the universal coefficient theorem
vanish. Thus the $KK$-class is determined by the induced $K_0$- and
$K_1$-maps \cite{RS}, so $KK(\phi)=KK(\Psi_{N,d}^{(s)})$.

If $d=0$, then $S_{N,0}$ is a constant matrix and
$$
 \Psi_{N,0}^{(s)}(M_s(C(\T)))
 =M_s\otimes C^*(S_{N,0})\subseteq M_{sN}(\C).
$$
The range is finite-dimensional and the fiber spectra are constant.
Hence $\SPV(\Psi_{N,0}^{(s)})=0$.

Suppose that $d\neq0$. For $z\in\T$,
$$
 \det(\lambda1_N-S_{N,d}(z))=\lambda^N-z^d.
$$
There are $N$ distinct roots of $\lambda^N=z^d$, and these are exactly
the eigenvalues of $S_{N,d}(z)$, each with multiplicity one.
Thus Definition~\ref{def:spectrum-spv} gives
$$
 \SP(\Psi_{N,d}^{(s)})_z=\{\lambda\in\T:\lambda^N=z^d\},
$$
with each root counted once, since the spectrum is defined using the
$e_{11}$-corner.

Given $z,w\in\T$, choose $t\in[-\pi,\pi]$ such that
$$
 z^d=e^{it}w^d.
$$
Multiplication by $e^{it/N}$ matches the spectrum at $w$ with the
spectrum at $z$. Therefore,
$$
 d_N\bigl(\SP(\Psi_{N,d}^{(s)})_z,\SP(\Psi_{N,d}^{(s)})_w\bigr)
 \leq |e^{it/N}-1|
 =2\left|\sin\frac{t}{2N}\right|
 \leq2\sin\frac{\pi}{2N}.
$$
Taking the supremum over $z,w\in\T$ proves the estimate.
\end{proof}

We now construct the model map for a homomorphism between circle algebras.
Let $\phi:B\to C$ be a unital homomorphism, and write
$$
 B=\bigoplus_{i=1}^{m}B_i,
 \qquad C=\bigoplus_{j=1}^{n}C_j,
$$
where each $B_i$ is either $M_{s_i}(C(\T))$ or $M_{s_i}$, and each
$C_j$ is either $M_{t_j}(C(\T))$ or $M_{t_j}$.

Set $q_{ij}=\phi^{i,j}(1_{B_i})$. Suppose that $B_i$ and $C_j$ are
circle blocks and $q_{ij}\ne0$. By Lemma~\ref{lem:straightening},
there is a $C(\T)$-linear isomorphism
$$
 \alpha_{i,j}:q_{ij}C_jq_{ij}
 \longrightarrow M_{s_i}\otimes M_{N_{ij}}(C(\T)),
$$
where $N_{ij}=\rank(q_{ij})/s_i$ is the multiplicity of
$\alpha_{i,j}\circ\phi^{i,j}$. Let $d_{ij}$ be its degree.
The homomorphism
$$
 \alpha_{i,j}^{-1}\circ\Psi_{N_{ij},d_{ij}}^{(s_i)}
 :B_i\longrightarrow q_{ij}C_jq_{ij}
$$
has the same multiplicity and degree. If $d_{ij}=0$, it has finite-dimensional range. If $B_i$ is a circle block and $C_j$ is finite-dimensional, we retain $\phi^{i,j}$. We also retain $\phi^{i,j}$
whenever $B_i$ is finite-dimensional.

Thus, for $b=\bigoplus_{i=1}^{m}b_i\in B$, define
$$
 \psi^{i,j}(b_i)=
 \begin{cases}
  0,&q_{ij}=0,\\[2mm]
  \alpha_{i,j}^{-1}\circ\Psi_{N_{ij},d_{ij}}^{(s_i)}(b_i),
    &q_{ij}\ne0\text{ and }B_i,C_j\text{ are circle blocks},\\[2mm]
  \phi^{i,j}(b_i),
    &q_{ij}\ne0,\ B_i\text{ or }C_j\text{ is finite-dimensional},
 \end{cases}
$$
and put
$$
 \psi(b)=\bigoplus_{j=1}^{n}
 \left(\sum_{i=1}^{m}\psi^{i,j}(b_i)\right).
$$
Each partial map has support $q_{ij}$. For fixed $j$, these supports
are mutually orthogonal and sum to $1_{C_j}$. Thus $\psi:B\to C$ is
a unital homomorphism, which we call a \emph{model map} for $\phi$.

For every $(i,j)$, we have
$$
 K_*(\psi^{i,j})=K_*(\phi^{i,j}).
$$
Lemma~\ref{lem:full-model} also gives
$$
 KK(\psi^{i,j})=KK(\phi^{i,j}),
 \qquad KK(\psi)=KK(\phi).
$$
If $B_i$ and $C_j$ are circle blocks and $q_{ij}\ne0$, the
identification $\alpha_{i,j}$ preserves fiber spectra. Hence
Lemma~\ref{lem:full-model} implies
$$
 \SPV(\psi^{i,j})\leq2\sin\frac{\pi}{2N_{ij}}.
$$
Its spectral variation is zero when $d_{ij}=0$. The partial maps
involving a finite-dimensional block also have spectral variation zero.
Consequently,
$$
 \SPV(\psi)\leq
 \max\left\{2\sin\frac{\pi}{2N_{ij}}:
 \begin{array}{c}
 B_i,C_j\text{ are circle blocks},\\
 q_{ij}\ne0,\ d_{ij}\ne0
 \end{array}\right\},
$$
where the maximum is defined to be zero if the indexing set is empty.
All other partial components have finite-dimensional range.

\section{The three-corner construction}

We construct the three finite-dimensional algebras for a cyclic map,
including the case of composition with an outer homomorphism. We then combine
the constructions over the partial maps of a model map.

\begin{lemma}\label{lem:two-cut}
Let
$$
 \alpha:M_{sN}(C(\T))\longrightarrow pM_t(C(\T))p
$$
be a unital homomorphism, and put
$\Theta=\alpha\circ\Psi_{N,d}^{(s)}$ (see Definition \ref{def:cyclic-map}).  Let $L\geq1$ be an integer, and assume that
$N>8L+4$.  Then there are finite-dimensional unital subalgebras
$C,D,E\subseteq pM_t(C(\T))p$ and a positive contraction $h\in E$ such
that $E\subseteq C\cap D$ and, for every Laurent polynomial
$P=\sum_{|\ell|\leq L}a_\ell\otimes\zeta^\ell$ with each $a_\ell\in M_s$, we have
$$
 h\Theta(P)\in C,
 \qquad
 (1-h)\Theta(P)\in D,
 \qquad
 h(1-h)\Theta(P)\in E,
$$
while
$$
 \|[h,\Theta(P)]\|
 \leq\frac{4}{N}\sum_{|\ell|\leq L}|\ell|\,\|a_\ell\|.
$$
\end{lemma}

\begin{proof}
Since $\alpha$ is unital, every fiber representation of $M_{sN}$ has
dimension a multiple of $sN$; hence $\rank(p)=sNR$ for some $R\geq1$.
By Lemma~\ref{lem:straightening}, we may therefore identify
$$
 pM_t(C(\T))p
 \cong M_s\otimes M_N\otimes M_R(C(\T))
$$
so that $\alpha$ satisfies
$$
\alpha(a\otimes1_{\T})=a\otimes1_R,\qquad a\in M_{sN}(\C),
$$ and
$$
 \alpha(1_{sN}\otimes\zeta)=1_{sN}\otimes V
$$
for a unitary $V\in M_R(C(\T))$.

For every $\ell\in\mathbb Z$,
$$
\alpha(a_\ell\otimes S_{N,d}^{\ell})=\alpha(a_\ell\otimes 1_N)\cdot\alpha(1_s\otimes S_{N,d}^{\ell}).
$$
Let
$$
 T:=\sum_{k=1}^{N-1}e_{k,k+1}\otimes1_R+e_{N,1}\otimes V^d
 =
 \begin{pmatrix}
 0&1_R&0&\cdots&0\\
 0&0&1_R&\ddots&\vdots\\
 \vdots&\ddots&\ddots&\ddots&0\\
 0&\cdots&0&0&1_R\\
 V^d&0&\cdots&0&0
 \end{pmatrix}.
$$
Then $\alpha(1_s\otimes S_{N,d})=1_s\otimes T$, and
$$
 \Theta(P)=\alpha\circ\Psi_{N,d}^{(s)}(P) =\sum_{|\ell|\leq L}a_\ell\otimes T^\ell.
$$

The only entry of $T$ that depends on the outer homomorphism is the
entry $V^d$ across the cyclic boundary between $N$ and $1$. We will
choose $h$ to vanish near this boundary. A diagonal conjugation
moves the same entry to a second boundary, between $m$ and $m+1$,
where $h$ will equal one. The remaining transition regions lie
inside the two matrix blocks defining $E$. This explains the three
exact containments before we verify them by matrix multiplication.

 For
$1\leq \ell\leq L$, the powers of $T$ have the expressions below:
\begin{align*}
 T^\ell
 &=\sum_{j=1}^{N-\ell}e_{j,j+\ell}\otimes1_R
   +\sum_{j=N-\ell+1}^{N}e_{j,j+\ell-N}\otimes V^d,\\
 T^{-\ell}
 &=\sum_{j=\ell+1}^{N}e_{j,j-\ell}\otimes1_R
   +\sum_{j=1}^{\ell}e_{j,j-\ell+N}\otimes (V^d)^*.
\end{align*}

Let $m=\lfloor N/2\rfloor$ and put
$$
 g_N:=
 \diag(\underbrace{(V^d)^*,\ldots,(V^d)^*}_{m\ \mathrm{entries}},
       \underbrace{1_R,\ldots,1_R}_{N-m\ \mathrm{entries}})
 =
 \begin{pmatrix}
  1_m\otimes (V^d)^*&0\\
  0&1_{N-m}\otimes1_R
 \end{pmatrix}.
$$
Then $g_N$ is a unitary in $M_N\otimes M_R(C(\T))$.  Let
$g=1_s\otimes g_N$.

Now define finite-dimensional algebras
\begin{align*}
 C&:=M_s\otimes M_N\otimes\C1_R,\\
 D&:=gCg^*=M_s\otimes \left(g_N(M_N\otimes\C1_R)g_N^*\right),\\
 E&:=M_s\otimes
 \left(\begin{array}{cc}
   M_m & 0 \\
   0 & M_{N-m}
 \end{array}\right)
 \otimes\C1_R.
\end{align*}
Since $g$ commutes with $E$, we have $E\subseteq C\cap D$.
Although the entries of $g$ can vary with the point of $\T$, the
algebra $D=gCg^*$ is finite-dimensional: conjugation is an
isomorphism from the fixed matrix algebra $C$ onto $D$.
Moreover, every $e\in E$ satisfies $geg^*=e$, so the common
inclusion is exact and has no approximation error.

We now define $h$.  Since $m=\lfloor N/2\rfloor$ and $N>8L+4$, we have
$$
m-2L+1>\frac{N}{4},
 \quad
N-m-2L+1>\frac{N}{4}.
$$
For $1\leq j\leq N$, define the scalar $\lambda_j$ by
$$
 \lambda_j=
 \begin{cases}
   0, & \text{if } 1\leq j\leq L,\\[1mm]
   \dfrac{j-L}{m-2L+1}, & \text{if } L+1\leq j\leq m-L,\\[1mm]
    1, & \text{if } m-L+1\leq j\leq m+L,\\[1mm]
   \dfrac{N-L-j+1}{N-m-2L+1}, & \text{if } m+L+1\leq j\leq N-L,\\[1mm]
   0, & \text{if }  N-L+1\leq j\leq N.
 \end{cases}
$$

Then
$$
 |\lambda_{j+1}-\lambda_j|<\frac4N
 \quad (1\leq j\leq N-1).
$$

Put
$$
 H:=\diag(\lambda_1,\ldots,\lambda_N)\otimes1_R
$$
and define
$$
 h=1_s\otimes H.
$$
Then $h$ is a positive contraction in $E$.

Next we check that
$$
 h\Theta(P)=\sum_{|\ell|\leq L}a_\ell\otimes HT^\ell\in C.
$$
For $1\leq \ell\leq L$, we have
\begin{align*}
  HT^\ell &=\left( \sum_{j=1}^{N}\lambda_j e_{j,j}\otimes1_R\right) \left(\sum_{j=1}^{N-\ell}e_{j,j+\ell}\otimes1_R
   +\sum_{j=N-\ell+1}^{N}e_{j,j+\ell-N}\otimes V^d\right) \\[3pt]
   & = \sum_{j=1}^{N-\ell}\lambda_j e_{j,j+\ell}\otimes1_R
   +\sum_{j=N-\ell+1}^{N}\lambda_j e_{j,j+\ell-N}\otimes V^d\\
   & = \sum_{j=1}^{N-\ell}\lambda_j e_{j,j+\ell}\otimes1_R
\end{align*}
and
\begin{align*}
  HT^{-\ell} &=\left( \sum_{j=1}^{N}\lambda_j e_{j,j}\otimes1_R\right) \left(\sum_{j=\ell+1}^{N}e_{j,j-\ell}\otimes1_R
   +\sum_{j=1}^{\ell}e_{j,j-\ell+N}\otimes (V^d)^*\right) \\[3pt]
   & = \sum_{j=\ell+1}^{N}\lambda_j e_{j,j-\ell}\otimes1_R
   +\sum_{j=1}^{\ell}\lambda_j e_{j,j-\ell+N}\otimes (V^d)^*\\
   & = \sum_{j=\ell+1}^{N}\lambda_j e_{j,j-\ell}\otimes1_R\,.
\end{align*}
Together with $HT^0=H$, we get that $HT^\ell\in M_N\otimes\C1_R$ for every $|\ell|\leq L$.
Therefore,
$$
 h\Theta(P)=\sum_{|\ell|\leq L}a_\ell\otimes HT^\ell\in C.
$$

We now estimate the commutator.  For $1\leq \ell\leq L$, a similar
calculation gives
\begin{align*}
  T^\ell H &= \left(\sum_{j=1}^{N-\ell}e_{j,j+\ell}\otimes1_R
   +\sum_{j=N-\ell+1}^{N}e_{j,j+\ell-N}\otimes V^d\right) \left( \sum_{j=1}^{N}\lambda_j e_{j,j}\otimes1_R\right) \\[3pt]
   & = \sum_{j=1}^{N-\ell}\lambda_{j+\ell} e_{j,j+\ell}\otimes1_R
\end{align*}
and
\begin{align*}
  T^{-\ell}H &= \left(\sum_{j=\ell+1}^{N}e_{j,j-\ell}\otimes1_R
   +\sum_{j=1}^{\ell}e_{j,j-\ell+N}\otimes (V^d)^*\right) \left( \sum_{j=1}^{N}\lambda_j e_{j,j}\otimes1_R\right)\\[3pt]
   & = \sum_{j=\ell+1}^{N}\lambda_{j-\ell} e_{j,j-\ell}\otimes1_R\,.
\end{align*}
Here the wrap-around terms vanish because their right-hand column indices
belong to $\{1,\ldots,\ell\}$ in the first formula and to
$\{N-\ell+1,\ldots,N\}$ in the second, and $\lambda_j=0$ on both sets.
It follows that
$$
[H,T^\ell]=\sum_{j=1}^{N-\ell}(\lambda_j-\lambda_{j+\ell}) e_{j,j+\ell}\otimes1_R.
$$

Moreover,
$$
 \|[H,T^\ell]\|^2
 =\|[H,T^\ell]^*[H,T^\ell]\|
 =\max_{1\leq j\leq N-\ell}|\lambda_j-\lambda_{j+\ell}|^2.
$$
The estimate $|\lambda_{j+1}-\lambda_j|<4/N$ and telescoping give
$|\lambda_j-\lambda_{j+\ell}|<4\ell/N$.  Therefore,
$$
 \|[H,T^\ell]\|=\max_{1\leq j\leq N-\ell}
 |\lambda_j-\lambda_{j+\ell}| \leq\frac{4|\ell|}{N}.
$$
Similarly, we have
$$
 \|[H,T^{-\ell}]\|=\max_{\ell+1\leq j\leq N}
 |\lambda_j-\lambda_{j-\ell}| \leq\frac{4|\ell|}{N}.
$$

Together with $[H,T^0]=0$, this gives
$$
 \|[h,\Theta(P)]\|
 =\|\sum_{|\ell|\leq L}a_\ell\otimes[H,T^\ell]\|\leq\frac{4}{N}\sum_{|\ell|\leq L}|\ell|\,\|a_\ell\|.
$$

It remains to prove the containments in $D$ and $E$.  Direct
multiplication gives
$$
 \widetilde T:=g_N^*Tg_N
 =\sum_{\substack{1\leq k\leq N-1\\k\neq m}}
   e_{k,k+1}\otimes1_R
  +e_{m,m+1}\otimes V^d+e_{N,1}\otimes1_R.
$$
Equivalently,
$$
 \widetilde T=
 \begin{pmatrix}
 0&1_R&&&&&&\\
 &\ddots&\ddots&&&&&\\
 &&0&1_R&&&&\\
 &&&0&V^d&&&\\
 &&&&0&1_R&&\\
 &&&&&0&\ddots&\\
 &&&&&&\ddots&1_R\\
 1_R&&&&&&&0
 \end{pmatrix},
$$
where $V^d$ is the $(m,m+1)$ entry.
Let $[r]_N$ denote the unique element of $\{1,\ldots,N\}$ congruent to
$r$ modulo $N$.  For $1\leq\ell\leq L$, we have
\begin{align*}
 \widetilde T^\ell
 &=\sum_{\substack{1\leq j\leq N\\j\notin\{m-\ell+1,\ldots,m\}}}
       e_{j,[j+\ell]_N}\otimes1_R
   +\sum_{j=m-\ell+1}^{m}e_{j,j+\ell}\otimes V^d,\\
 \widetilde T^{-\ell}
 &=\sum_{\substack{1\leq j\leq N\\j\notin\{m+1,\ldots,m+\ell\}}}
       e_{j,[j-\ell]_N}\otimes1_R
   +\sum_{j=m+1}^{m+\ell}e_{j,j-\ell}\otimes(V^d)^*.
\end{align*}
The rows in the second sums lie in $\{m-L+1,\ldots,m+L\}$, where
$1-\lambda_j=0$.
Equivalently,
\begingroup
\setcounter{MaxMatrixCols}{12}
\setlength{\arraycolsep}{3pt}

\[
\widetilde T^\ell=
\begin{pmatrix}
0 & \cdots & 0 & 1_R
  & & & & & & & & \\
  & \ddots & & & \ddots
  & & & & & & & \\
  & & 0 & & & 1_R
  & & & & & & \\
  & & & 0 & & & V^d
  & & & & & \\
  & & & & \ddots & & & \ddots
  & & & & \\
  & & & & & 0 & & & V^d
  & & & \\
  & & & & & & 0 & & & 1_R
  & & \\
  & & & & & & & \ddots & & & \ddots
  & \\
  & & & & & & & & 0 & & & 1_R \\
1_R & & & & & & & & & 0 & & \\
  & \ddots & & & & & & & & & \ddots & \\
  & & 1_R & & & & & & & & & 0
\end{pmatrix},
\]

\[
\widetilde T^{-\ell}=
\begin{pmatrix}
0 & & & & & & & & & 1_R & & \\
  & \ddots & & & & & & & & & \ddots & \\
  & & 0 & & & & & & & & & 1_R \\
1_R & & & 0 & & & & & & & & \\
  & \ddots & & & \ddots & & & & & & & \\
  & & 1_R & & & 0 & & & & & & \\
  & & & (V^d)^* & & & 0 & & & & & \\
  & & & & \ddots & & & \ddots & & & & \\
  & & & & & (V^d)^* & & & 0 & & & \\
  & & & & & & 1_R & & & 0 & & \\
  & & & & & & & \ddots & & & \ddots & \\
  & & & & & & & & 1_R & 0 & \cdots & 0
\end{pmatrix}.
\]
\endgroup
In the first matrix, the nonzero upper band is the
$\ell$th superdiagonal. Its entries are $V^d$ in positions
$(j,j+\ell)$ for $m-\ell+1\leq j\leq m$, and $1_R$
elsewhere. The wrap-around entries are $1_R$ in positions
$(N-\ell+j,j)$ for $1\leq j\leq\ell$.
The second matrix is the adjoint of the first. Hence
$$
 (1_{NR}-H)\widetilde T^\ell\in M_N\otimes\C1_R
 \qquad(0<|\ell|\leq L).
$$
The same inclusion holds for $\ell=0$.  Since $g_NH=Hg_N$,
\begin{align*}
 g^*(1_{sNR}-h)\Theta(P)g
 &=(1_s\otimes g_N^*)(1_s\otimes(1_{NR}-H))
       \Theta(P)(1_s\otimes g_N)\\
 &=\sum_{|\ell|\leq L}a_\ell\otimes
       (1_{NR}-H)\widetilde T^\ell\in C.
\end{align*}
Thus $(1_{sNR}-h)\Theta(P)\in gCg^*=D$.

Finally, the diagonal entries of $H(1_{NR}-H)$ vanish outside
$$
 J:=\{L+1,\ldots,m-L\}\cup\{m+L+1,\ldots,N-L\}.
$$
If $j$ belongs to the first interval and $|\ell|\leq L$, then
$1\leq j+\ell\leq m$.  If $j$ belongs to the second interval, then
$m+1\leq j+\ell\leq N$.  The formulas for $T^\ell$ give
$$
 \begin{aligned}
 H(1_{NR}-H)T^\ell
 &=\sum_{j\in J}\lambda_j(1-\lambda_j)e_{j,j+\ell}\otimes1_R\\
 &\in(M_m\oplus M_{N-m})\otimes\C1_R
 \qquad(|\ell|\leq L).
 \end{aligned}
$$
Hence
$$
 h(1_{sNR}-h)\Theta(P)
 =\sum_{|\ell|\leq L}a_\ell\otimes H(1_{NR}-H)T^\ell\in E.
$$
\end{proof}

\begin{proposition}\label{prop:local-model}
Let $B=\bigoplus_{i=1}^{m}B_i$ be a circle algebra, let $X$ be a finite
subset of its unit ball, and let $\varepsilon,\rho>0$. There is
$\delta>0$ such that every unital homomorphism
$\phi:B\to C=\bigoplus_{j=1}^{n}C_j$ into a circle algebra with
$\SPV(\phi)<\delta$ admits a model map $\psi$ satisfying
$$
 KK(\psi)=KK(\phi),\qquad \SPV(\psi)<\rho.
$$
Moreover, there are finite-dimensional unital subalgebras
$C_0,D_0,E_0\subseteq C$ and a positive contraction $h\in E_0$, with
$E_0\subseteq C_0\cap D_0$, such that, for every $x\in X$,
$$
 \|[h,\psi(x)]\|<\varepsilon,\quad
 h\psi(x)\in_\varepsilon C_0,\quad
 (1-h)\psi(x)\in_\varepsilon D_0,\quad
 h(1-h)\psi(x)\in_\varepsilon E_0.
$$
\end{proposition}

\begin{proof}
If $X$ is empty, there is nothing to approximate; we may enlarge it
to $\{0\}$. For every $x=\bigoplus_{i=1}^m x_i\in X$ and every
circle block $B_i$, choose a Laurent polynomial $P_{x,i}$ such that
$$
 P_{x,i}=\sum_{|\ell|\leq L}a_{x,i,\ell}\otimes\zeta^\ell,
 \qquad
 \|x_i-P_{x,i}\|<\frac{\varepsilon}{4}.
$$
Choose a common integer $L\geq1$ for this finite collection, adding
zero coefficients when necessary. For a finite-dimensional block,
put $P_{x,i}=x_i$. Write $P_x=\bigoplus_{i=1}^m P_{x,i}$ and set
$$
 K=\max_{\substack{x\in X\\B_i\text{ a circle block}}}
       \sum_{|\ell|\leq L}|\ell|\,\|a_{x,i,\ell}\|,
$$
with $K=0$ if there are no circle blocks. Thus the finite-dimensional
source summands require no polynomial approximation.

Choose an integer
$$
 M>\max\left\{8L+4,\frac{8K}{\varepsilon},\frac{\pi}{\rho}\right\}
$$
and
$$
 0<\delta<2\sin\frac{\pi}{2M}<1.
$$
Let $\psi$ be a model map for $\phi$, and set
$q_{ij}=\phi^{i,j}(1_{B_i})=\psi^{i,j}(1_{B_i})$.

Define two index sets
\begin{align*}
 J_0&=\{(i,j):1\leq i\leq m,\ 1\leq j\leq n,\
                \psi^{i,j}\text{ has finite-dimensional range}\},\\
 J_1&=\{(i,j):1\leq i\leq m,\ 1\leq j\leq n,\
                \psi^{i,j}\text{ has nonzero degree}\}.
\end{align*}
By the model construction, these sets partition the pairs $(i,j)$.

For $(i,j)\in J_0$, the spectral variation of $\psi^{i,j}$ is zero.
Define
$$
 h_{ij}=0,\qquad C_{ij}=E_{ij}=\C q_{ij},
 \qquad D_{ij}=\psi^{i,j}(B_i).
$$
The three containments are exact and the commutator is zero.

For $(i,j)\in J_1$, let $N_{ij}$ be the multiplicity of
$\phi^{i,j}$. Apply Lemma~\ref{lem:spectral-cover} under the
identification $\alpha_{i,j}$ to obtain
$$
 N_{ij}\geq\frac{\pi}{2\arcsin(\delta/2)}>M.
$$
Hence
$$
 \SPV(\psi^{i,j})\leq2\sin\frac{\pi}{2N_{ij}}
 \leq\frac{\pi}{N_{ij}}<\rho.
$$
Since $N_{ij}>8L+4$, Lemma~\ref{lem:two-cut}, applied with the outer
map $\alpha_{i,j}^{-1}$, gives $C_{ij},D_{ij},E_{ij}\subseteq q_{ij}C_jq_{ij}$
and $h_{ij}\in E_{ij}$, with $E_{ij}\subseteq C_{ij}\cap D_{ij}$, such that
\begin{align*}
 h_{ij}\psi^{i,j}(P_{x,i})&\in C_{ij},\\
 (q_{ij}-h_{ij})\psi^{i,j}(P_{x,i})&\in D_{ij},\\
 h_{ij}(q_{ij}-h_{ij})\psi^{i,j}(P_{x,i})&\in E_{ij},
\end{align*}
and the commutator norm is at most $4K/N_{ij}<\varepsilon/2$.

Take the orthogonal sums of these algebras and contractions:
\begin{align*}
 h&=\bigoplus_{j=1}^{n}\left(\sum_{i=1}^{m}h_{ij}\right),
 &C_0&=\bigoplus_{j=1}^{n}\left(\sum_{i=1}^{m}C_{ij}\right),\\
 D_0&=\bigoplus_{j=1}^{n}\left(\sum_{i=1}^{m}D_{ij}\right),
 &E_0&=\bigoplus_{j=1}^{n}\left(\sum_{i=1}^{m}E_{ij}\right).
\end{align*}
The supports sum to $1_C$, so $C_0,D_0,E_0$ are unital,
$h\in E_0$, and $E_0\subseteq C_0\cap D_0$.
Also $\SPV(\psi)<\rho$. For every $x\in X$,
$$
 \|[h,\psi(x)]\|
 \leq\|[h,\psi(P_x)]\|+2\|x-P_x\|
 <\frac{\varepsilon}{2}+2\cdot\frac{\varepsilon}{4}
 =\varepsilon.
$$
Each of the three containment tolerances is at most
$\|x-P_x\|<\varepsilon/4$. This concludes the proof.
\end{proof}

\section{Main results}

We replace the connecting maps by model maps with increasingly accurate
three-corner estimates. The resulting system has real rank zero and is
$KK$-shape equivalent to the original system.
We recall the definition and the isomorphism theorem from
\cite[Section~7]{DG} and \cite[Definition~6.10 and Theorem~8.2]{ALZ}.

\begin{definition}\label{def:KK-shape}
Let $A=\varinjlim(A_n,\phi_{n,m})$ and
$B=\varinjlim(B_n,\psi_{n,m})$ be unital inductive limits with unital
connecting maps. The two systems are \emph{$KK$-shape equivalent} if
there are strictly increasing sequences of integers $(k_i)$ and $(l_i)$
and unital homomorphisms
$$
 \xi_i:A_{k_i}\longrightarrow B_{l_i},
 \qquad \eta_i:B_{l_i}\longrightarrow A_{k_{i+1}}
$$
such that
$$
 KK(\eta_i\circ\xi_i)=KK(\phi_{k_i,k_{i+1}})
 \quad\text{in }KK(A_{k_i},A_{k_{i+1}})
$$
and
$$
 KK(\xi_{i+1}\circ\eta_i)=KK(\psi_{l_i,l_{i+1}})
 \quad\text{in }KK(B_{l_i},B_{l_{i+1}}).
$$
\end{definition}

\begin{proposition}[{\cite[Theorem~7.3(b)]{DG},
\cite[Theorem~8.2]{ALZ}}]\label{prop:KK-shape}
Let $A$ and $B$ be unital AT algebras of real rank zero, with presentations
$$
 A=\varinjlim(A_n,\phi_{n,m}),\qquad
 B=\varinjlim(B_n,\psi_{n,m}).
$$
If their defining systems are $KK$-shape equivalent, then $A\cong B$.
\end{proposition}

\begin{proof}
In the present AT setting, this follows directly from
Theorem~\ref{thm:AT-classification}. Put
$F_i=K_*(\xi_i)$ and $G_i=K_*(\eta_i)$, using the maps from
Definition~\ref{def:KK-shape}. The two $KK$ identities imply
$$
 G_iF_i=K_*(\phi_{k_i,k_{i+1}}),\qquad
 F_{i+1}G_i=K_*(\psi_{l_i,l_{i+1}}).
$$
In particular,
$$
 F_{i+1}K_*(\phi_{k_i,k_{i+1}})
 =K_*(\psi_{l_i,l_{i+1}})F_i,
$$
and the analogous compatibility holds for the $G_i$.
They therefore induce mutually inverse graded group maps on the
inductive limits. Each $F_i$ and $G_i$ comes from a unital
homomorphism, so it preserves the graded positive cone and the order
unit. Continuity of the scaled ordered $K_*$-group shows that the
limit maps have the same properties. They are thus inverse
isomorphisms of the invariants in
Definition~\ref{def:graded-invariant}. Both algebras have real rank
zero, so the AT classification theorem gives $A\cong B$.
\end{proof}

The following is the finite-direct-sum version of
\cite[Proposition~1.4.10]{EG}.
\begin{lemma}
\label{lem:spectral-composition}
Let $A$ and $B$ be circle algebras, and let
$\phi:A\to B$ be a unital homomorphism.
For every $\varepsilon>0$, there is $0<\rho<\varepsilon$ such that,
for every circle algebra $C$ and every unital homomorphism
$\eta:B\to C$,
$$
 \SPV(\eta)<\rho
 \quad\Longrightarrow\quad
 \SPV(\eta\circ\phi)<\varepsilon.
$$
The number $\rho$ depends only on $\varepsilon$ and the fixed map
$\phi$. In particular, it is independent of $C$,
of $\eta$, and of the multiplicities of the partial maps of $\eta$.
\end{lemma}

\begin{proof}[Proof of Theorem~\ref{thm:main}]
Write $A=\varinjlim(A_n,\phi_{n,m})$ with unital connecting maps,
and set $n_1=1$. We first construct a subsequence $n_1<n_2<\cdots$
and unital maps
$$
 \psi_{k,k+1}:B_k\longrightarrow B_{k+1},
 \qquad B_k=A_{n_k},
$$
with small spectral variation and
$$
 KK(\psi_{k,k+1})=KK(\phi_{n_k,n_{k+1}}).
$$
Write $\psi_{i,k}$ for the composites and
$\psi_{k,k}=\mathrm{id}_{B_k}$.

Suppose that $B_1,\ldots,B_k$ and the preceding connecting maps have
been chosen. As each $B_i$ is chosen, fix a dense sequence
$\{x_{i,l}:l\geq1\}$ in its unit ball. Set
$$
 X_k=\bigcup_{i=1}^{k}
       \{\psi_{i,k}(x_{i,l}):1\leq l\leq k\}\subseteq B_k.
$$
For each $1\leq i\leq k$, apply
Lemma~\ref{lem:spectral-composition} to the map
$\psi_{i,k}:B_i\to B_k$ with $\varepsilon=2^{-k}$.
This gives a number $0<\rho_{i,k}<2^{-k}$ such that, for every
circle algebra $C$ and every unital homomorphism $\eta:B_k\to C$,
$$
 \SPV(\eta)<\rho_{i,k}
 \quad\Longrightarrow\quad
 \SPV(\eta\circ\psi_{i,k})<2^{-k}.
$$
Set
$$
 \rho_k=\min_{1\leq i\leq k}\rho_{i,k}.
$$
Since the minimum is taken over finitely many positive numbers,
$0<\rho_k<2^{-k}$. Consequently the same tolerance works for
all the maps already chosen:
$$
 \SPV(\eta)<\rho_k
 \quad\Longrightarrow\quad
 \SPV(\eta\circ\psi_{i,k})<2^{-k}
 \quad(1\leq i\leq k).
$$
Here $\rho_k$ depends only on the maps already chosen and the
required bound $2^{-k}$, so it can be fixed before choosing the
next stage and its connecting map.

Apply Proposition~\ref{prop:local-model} to $X_k$ with tolerances
$2^{-k}$ and $\rho_k$, and let $\delta_k>0$ be the resulting constant.
By Theorem~\ref{thm:small-variation}, there is $n_{k+1}>n_k$ such that
$$
 \SPV(\phi_{n_k,n_{k+1}})<\delta_k.
$$
Choose $\psi_{k,k+1}$ to be a model map for $\phi_{n_k,n_{k+1}}$.
By Proposition~\ref{prop:local-model},
$$
 KK(\psi_{k,k+1})=KK(\phi_{n_k,n_{k+1}}),
 \qquad \SPV(\psi_{k,k+1})<\rho_k.
$$
There are finite-dimensional unital subalgebras
$C_k,D_k,E_k\subseteq B_{k+1}$ and a positive contraction $h_k\in E_k$
with $E_k\subseteq C_k\cap D_k$ satisfying the four estimates of
Proposition~\ref{prop:local-model} on $X_k$ with tolerance $2^{-k}$.
This completes the induction. Set
$$
 B=\varinjlim(B_k,\psi_{k,l}).
$$
For each fixed $i$, the choice of $\rho_k$ gives
$$
 \SPV(\psi_{i,k+1})<2^{-k}\qquad(k\geq i).
$$
Hence $B$ has real rank zero by Theorem~\ref{thm:small-variation}.

We next identify the two limits. For each $k$, define unital
homomorphisms
$$
 \alpha_k=\mathrm{id}:A_{n_k}\longrightarrow B_k,
 \qquad
 \beta_k=\phi_{n_k,n_{k+1}}:B_k\longrightarrow A_{n_{k+1}},
$$
using $B_k=A_{n_k}$. Then
\begin{align*}
 KK(\beta_k\circ\alpha_k)&=KK(\phi_{n_k,n_{k+1}}),\\
 KK(\alpha_{k+1}\circ\beta_k)
 &=KK(\phi_{n_k,n_{k+1}})=KK(\psi_{k,k+1}).
\end{align*}
In particular, the following diagram commutes in $KK$:
$$
 \begin{CD}
 A_{n_k} @>{KK(\phi_{n_k,n_{k+1}})}>> A_{n_{k+1}}\\
 @V{KK(\alpha_k)}VV @VV{KK(\alpha_{k+1})}V\\
 B_k @>{KK(\psi_{k,k+1})}>> B_{k+1}.
 \end{CD}
$$
Thus $(A_{n_k},\phi_{n_k,n_l})$ and $(B_k,\psi_{k,l})$ are
$KK$-shape equivalent. Both limits have real rank zero, so
Proposition~\ref{prop:KK-shape} gives $A\cong B$.

It remains to verify the three-corner estimates in $B$.
Let $X$ be a finite subset of the unit ball of $B$, and fix
$\varepsilon>0$. Write $\iota_i:B_i\to B$ for the canonical maps.
Choose a stage $i$ and indices $l_x$ such that
$$
 y_x:=\iota_i(x_{i,l_x}),\qquad
 \|x-y_x\|<\frac{\varepsilon}{4}\quad(x\in X).
$$
Here is the unit-ball approximation in detail, including the case
of noninjective connecting maps. Given a small $\tau>0$, approximate
each $x$ by $\iota_j(a_x)$ within $\tau$ at a common stage $j$.
Then $\|\iota_j(a_x)\|<1+\tau$. Since
$$
 \|\iota_j(a_x)\|=\lim_{r\to\infty}\|\psi_{j,r}(a_x)\|,
$$
there is a common later stage $r$ at which all these representatives
have norm less than $1+2\tau$. Put
$b_x=\psi_{j,r}(a_x)/(1+2\tau)$, so $\|b_x\|\leq1$ and
$\|x-\iota_r(b_x)\|<3\tau$.
Choose $x_{r,l_x}$ within $\tau$ of $b_x$.
Taking $\tau<\varepsilon/16$ gives the required approximation with
$i=r$. Thus, no injectivity assumption is used.

Choose $k\geq\max\{i,l_x:x\in X\}$ with $2^{-k}<\varepsilon/2$.
Then $\psi_{i,k}(x_{i,l_x})\in X_k$. Apply $\iota_{k+1}$ to
$C_k,D_k,E_k,h_k$, and denote their images by $C,D,E,h$.
The algebras remain finite-dimensional and unital, and
$$
 h\in E,\qquad E\subseteq C\cap D.
$$
Contractivity gives the four estimates in Proposition \ref{prop:criterion} for
$y_x$ with tolerance $2^{-k}$. Consequently,
\begin{align*}
 \|[h,x]\|&<2^{-k}+2\|x-y_x\|<\varepsilon,\\
 hx&\in_{\,2^{-k}+\varepsilon/2}C,\\
 (1-h)x&\in_{\,2^{-k}+\varepsilon/2}D,\\
 h(1-h)x&\in_{\,2^{-k}+\varepsilon/2}E.
\end{align*}
By
Proposition~\ref{prop:criterion}, $B$ has complexity rank at most one. Since $A$ is isomorphic to $B$, the same is true for $A$.
\end{proof}

\section*{AI statement}
 During the development of this work, the authors use GPT 6 Astra for the construction in Lemma 4.1. The authors have checked and substantially revised all arguments and takes full responsibility for the final version.

\end{document}